\documentclass[final]{dmtcs-episciences}

\usepackage[utf8]{inputenc}
\usepackage[T1]{fontenc}
\usepackage[english]{babel}

\title{On the multi-robber damage number\thanks{This work has been supported by COST Action CA22145 GameTable.}}
\author{Miloš Stojaković\affiliationmark{1}\thanks{Partly supported by the Science Fund of the Republic of Serbia, Grant \#7462: Graphs in Space and Time: Graph Embeddings for Machine Learning in Complex Dynamical Systems (TIGRA), and partly supported by the Ministry of Science, Technological Development and Innovation of the Republic of Serbia (grants 451-03-33/2026-03/200125 \& 451-03-34/2026-03/200125).} \and Lasse Wulf\affiliationmark{2}\thanks{Supported by the Austrian Science Fund (FWF):W1230 and by the Carlsberg Foundation CF21-0302 ``Graph Algorithms with Geometric Applications''.
}}
\affiliation{Department of Mathematics and Informatics, Faculty of Sciences, University of Novi Sad, Novi Sad, Serbia.\\
Department of Mathematics and Computer Science, University of Southern Denmark, Vejle, Denmark.}
\date{}

\usepackage{amsmath}
\usepackage{amsfonts}
\usepackage{amssymb}
\usepackage{amsthm}

\usepackage{thmtools}

\usepackage{graphicx}

\usepackage{hyperref}
\usepackage[capitalize, noabbrev]{cleveref}

\newtheorem{theorem}{Theorem}[section]
\newtheorem{definition}[theorem]{Definition} 
\newtheorem{lemma}[theorem]{Lemma}
\newtheorem{proposition}[theorem]{Proposition}

\newtheorem{observation}[theorem]{Observation}	
\newtheorem{conjecture}[theorem]{Conjecture}

\newcommand{\set}[1]{\{#1\}}

\newcommand{\N}{\mathbb{N}}

\DeclareMathOperator{\dmg}{dmg}
\DeclareMathOperator{\guard}{\textsc{guard}}

\usepackage{xcolor}


\begin{document}
\publicationdata{vol. 28:3}{2026}{16}{10.46298/dmtcs.16953}{2025-11-18; 2025-11-18; 2026-06-23}{2026-09-25}
\emergencystretch 3em

\maketitle

\begin{abstract}
We study a variant of the Cops and Robbers game on graphs in which the robbers damage the visited vertices, aiming to maximize the number of damaged vertices. For that game with one cop against $s$ robbers a conjecture was made by Carlson, Halloran and Reinhart that the cop can save three vertices from being damaged as soon as the maximum degree of the base graph is at least $\binom{s}{2} + 2$.
We are able to verify the conjecture and prove that it is tight once we add the assumption that the base graph is triangle free.
We also study the game without that assumption, disproving the conjecture in full generality and further attempting to locate the smallest maximum degree of a base graph which guarantees that the cop can save three vertices against $s$ robbers. We  show that this number is between $2\binom{s}{2} - 3$ and  $2\binom{s}{2} + 1$.
Furthermore, after  the game has been previously studied with one cop and multiple robbers, as well as with one robber and multiple cops, we initiate the study of the game with two cops and two robbers. In the case when the base graph is a cycle we determine the exact number of damaged vertices. Additionally, when the base graph is a path we provide bounds that differ by an additive constant.
\end{abstract}


\section{Introduction}
The \emph{Cops and Robbers game} is a famous graph-based pursuit-evasion game with perfect information, where multiple cops try to catch a single robber. It was introduced independently by Quilliot~\cite{quilliot1978jeux} and by Nowakowski and Winkler~\cite{nowakowski1983vertex}. A book by Bonato and Nowakovski~\cite{bonato2011game} covers this general subject in great detail. 

The Cops and Robbers game is famous for the long standing Meyniel's conjecture, see~\cite{bonato2012Meyniel}, which states that $\Theta(\sqrt{n})$ cops always suffice to catch a robber on any graph. The best known bounds today on the number of required cops is only known to lie between $\Omega(\sqrt{n})$ \cite{pralat2010does} and $O(n/2^{(1+o(1))\sqrt{\log n}})$ \cite{frieze2012variations, lu2012meyniel, scott2011bound}. 

Over the years, the game has seen many variants and suggested modifications. In several of these modifications, including the original game, the cops' goal is merely to catch the robber in finitely many moves. In such a setting it clearly makes no sense to add more robbers into the picture.

Cox and Sanai~\cite{cox2019damage} introduced a variant of the game that gives the robber a more active role than simply evading the cops: the robber tries to damage as many vertices as possible, by visiting them before getting captured, and the cops attempt to minimize this damage. While the damage variant was originally studied with one cop and one robber, it was later extended to multiple cops by Carlson, Halloran and Reinhart \cite{carlson2021damage}. Unlike the original game, in the damage variant it makes sense to introduce multiple robbers, which was done recently by Carlson, Eagleton, Geneson, Petrucci, Reinhart, and Sen \cite{carlson2022multi}. Finally, Huggan, Messinger and Porter \cite{huggan2024damage,cor-huggan2024damage} consider the damage number of a graph as a parameter, and analyzed how this parameter behaves with respect to the graph product.


For a formal definition of the problem, let $G$ be a graph and $s$ a positive integer. In round 0 the cop first chooses an initial vertex to occupy, and then each of the $s$ robbers chooses an initial vertex. Each subsequent round consists of a turn for the cop followed by a turn for the robbers, where each individual has the opportunity to (but is not required to) move to a neighboring vertex on their turn. If the cop ever occupies the same vertex
as a robber, we say that that robber is captured, and it is subsequently removed from the game. All information about the locations and movements of everyone is publicly available, i.e., this is a perfect information game.

Now it remains to add the concept of \emph{damage} to the setup. If a vertex $v$ is occupied by a robber at the end of a round and the robber is not caught in the \emph{following} round, then $v$ becomes damaged and stays damaged for the rest of the game. In this version of the game the cop is trying to minimize the number of damaged vertices, while the robbers play to damage as many vertices as possible. The \emph{$s$-robber damage number} of G, denoted $\text{dmg}(G;s)$, is the number of damaged vertices when both the cop and the $s$ robbers play optimally on $G$.

We note that this is not the only version of Cops and Robbers where the robbers' movement on the base graph is tracked. Kinnersley and Peterson~\cite{kinnersley2018cops} introduced the ``burning bridges'' variant, where the robbers actually change the base graph as the game is played  (unlike in our case), by removing edges that they traverse. See~\cite{herrman2022capture} for further progress. Clearly, adding more robbers in this variant could change the outcome, just as it might in our setting.

\paragraph{Our results.} In \cref{sec:protect-three-vertices}, we consider the setting with one cop and multiple robbers. We say that the cop \emph{saves} $k$ vertices, if $\dmg(G;s) \leq n - k$. In other words, assuming optimal play from both the cop and the robbers, we have that $k$ out of the $n$ vertices do not get damaged.
We can now ask how many vertices can the cop save? 

As it was observed by Carlson, Halloran and Reinhart~\cite{carlson2022multi}, it is always possible for a single cop to save two vertices (assuming the graph has at least one edge): 
if the cop simply patrols a fixed edge of the graph, none of its endpoints can ever get damaged. Therefore we trivially have that two vertices can always get saved. 
This begs the question: When is it possible to save three vertices? 

This question was studied thoroughly in~\cite{carlson2022multi} for the case of $s=2$ robbers. They showed that on any graph $G$ with $\Delta(G) \geq 3$ (that is, there is at least one vertex of degree at least 3) the cop can indeed save three vertices. They used this insight to give a full characterization of all the graphs where the cop can save three vertices against 2 robbers. 

For the case of more than two robbers, they proposed the following conjecture.

\begin{conjecture}[\cite{carlson2022multi}]
For all $s > 2$, if $G$ is a graph with $\Delta(G) \geq \binom{s}{2} + 2$, then the cop can save three vertices against $s$ robbers.
\end{conjecture}

It turns out that this conjecture is ``almost true'', as we are able to prove it with an additional condition, that the graph $G$ is triangle-free. In the general case, when $G$ is not necessarily triangle-free, we have found counterexamples that falsify the conjecture. It becomes true and (asymptotically) tight again, if the right side of the corresponding bound $\Delta(G) \geq \binom{s}{2} + 2$ is multiplied by a factor of 2.

Before we state our two main results, we introduce a notation for the two extremal values that we are after. 

\begin{definition}
For $s \geq 1$, let $\Delta_s$ denote the smallest integer such that for all graphs $G$ with $\Delta(G) \geq \Delta_s$, the cop can save three vertices against $s$ robbers. 

Further, let $\Delta'_s$ denote the smallest integer such that for all triangle-free graphs $G$ with $\Delta(G) \geq \Delta'_s$, the cop can save three vertices against $s$ robbers.
\end{definition}

In the case of triangle-free graphs we can determine the exact value of the critical number of robbers.
\begin{theorem} \label{t:main-prime}
For $s\geq 2$, we have $$\Delta'_s = \binom{s}{2} + 2.$$
\end{theorem}
For general graphs we get the following bounds for the critical number of robbers.
\begin{theorem}\label{t:main}
For $s\geq 2$, we have 
$$ 2\binom{s}{2} - 3 \leq \Delta_s \leq 2\binom{s}{2} + 1. $$
\end{theorem}

For the second part of this paper, we recall that Carlson, Halloran and Reinhart~\cite{carlson2021damage} considered the setting with a single cop and multiple robbers, while Carlson, Eagleton, Geneson, Petrucci, Reinhart, and Sen~\cite{carlson2022multi} considered the setting with a single robber and multiple cops.

In Section~\ref{sec:two-cops-two-robbers} we naturally broaden the study to the next feasible case, with two cops and two robbers.
This setting is considered on the two families of underlying graphs, cycles and paths. We define $\dmg(G;s,t)$ to be the damage number of the graph $G$ when $s$ robbers are facing $t$ cops. We let $C_n$ be the cycle on $n$ edges and $P_{n+1}$ be the path on $n+1$ vertices and $n$ edges, and we prove the following.

\begin{theorem}
\label{thm:damage-cycle}
    When two cops are playing against two robbers on a cycle, we have
\[ \dmg(C_n;2,2) = \left\lfloor \frac{2}{3}n - \frac{5}{2} \right\rfloor.  \]
\end{theorem}

This means that as $n$ approaches infinity the additive constant $5/2$ becomes negligible, and an optimal play asymptotically results in a fraction of $2/3$ of the cycle damaged. 

To get the exact value we performed a meticulous analysis supported by extensive technical calculations. Having this in mind, along with the fact that strategies on a cycle and on a path share similarities, we decided not to overload the paper by determining the value of $\dmg(P_{n+1}; 2,2)$ that precisely.  
For the path we obtain bounds that are an additive constant away from each other, showing that, asymptotically, just like in the case of the cycle, a fraction of $2/3$ of the path will be damaged by the robbers, assuming optimal play from both sides.

\begin{theorem}
\label{thm:damage-path}
    When two cops are playing against two robbers on a path, we have
    \[
    \dmg(P_{n+1}; 2, 2) = \frac{2n}{3} - \Theta(1).
    \]
\end{theorem}

Finally, let us note that, after the first version of this paper appeared, Gagarova~\cite{gagarova2025exploring} also investigated the multi-robber damage number. For $s$ robbers against one cop, all graphs on $n\geq 2s+2$ vertices with the smallest possible damage number are characterized. Furthermore, it was established that the sole cop can save $k>3$ vertices from damage against $s$ robbers, provided that $\Delta(G)\geq \binom{s}{2} + k - 1$, extending the upper bound we obtained for Theorem~\ref{t:main-prime}. The investigation in~\cite{gagarova2025exploring} is further broadened by deriving a number of bounds for special families of graphs.

\section{Can one cop protect three vertices?}
\label{sec:protect-three-vertices}
\subsection{Upper bounds}

\begin{lemma} \label{lemma:induced-star}
If the graph $G$ contains an induced star $K_{1,t}$ with $t \geq \binom{s}{2} + 2$, then the cop can save 3 vertices against $s$ robbers.
\end{lemma}
\begin{proof}
Assume $G$ contains an induced star $K_{1,t}$. 
We have to give a strategy for the cop which saves three vertices. 
Let $v$ be the central vertex of the star and let $W = \set{w_1,\dots,w_t}$ be the set of its leaves. 
Recall that every round consists of a cop move followed by the robbers moves. The strategy of the cop is to guard the central vertex $v$. Note that any robber that never enters any of the vertices of $W$ can obviously never damage any of them.

Let the $i_1$-th round be the first round where a robber enters $N(v)$. 
The cop immediately moves to catch this robber in round $i_1 + 1$ and afterwards moves back to $v$ in round $i_1 + 2$.
 Now, the cop repeats the same strategy: We let $i_2$ be the smallest integer satisfying $i_2 \geq i_1 + 2$ such that a robber enters $N(v)$ in round $i_2$.
  The cop catches this robber and afterwards moves back to vertex $v$, and so on.
  
  We claim that this strategy guarantees that vertex $v$ and at least two vertices from $W$ do not get damaged. 
  Indeed, vertex $v$ stays undamaged, because there are no two consecutive rounds where $v$ is not occupied by the cop. 
  Furthermore, the only rounds where some vertex in $W$ can get damaged are by definition the rounds $i_1, i_1+1, i_2, i_2+1, i_3, i_3+1$, and so on. 
  Observe that after round $i_r+1$ the last remaining robber is caught (of the robbers that ever enter $W$).
   Furthermore, for all $j=1,\dots, r$, we have that in the rounds $i_j, i_j+1$ at most $s-j$ robbers are not caught. In these two rounds, each one of them can damage at most one vertex from $W$, because no two vertices of $W$ are connected by an edge. In total, the robbers damage at most $(s-1) + (s-2) + \dots + 1 = \binom{s}{2}$ vertices, so at least 2 vertices from $W$ survive. Therefore the cop saved three vertices.
\end{proof}

The strategy described in the proof of \cref{lemma:induced-star} will be useful in other situations too, so we give it a name. We define the strategy $\guard(v)$ to be the strategy for the cop of initially starting at vertex $v$ and waiting until a robber enters $N(v)$, then catching this robber, and then immediately going back to $v$. (Even if there is a second robber in $N(v)$ that could potentially be caught, the cop still moves back to $v$.)

First we prove that the upper bound in Theorem~\ref{t:main-prime} holds.

\begin{proposition}
For all $s \geq 2$, we have $\Delta_s' \leq \binom{s}{2} + 2$.
\end{proposition}
\begin{proof}
Let $t := \binom{s}{2} + 2$. If $G$ is triangle-free and $\Delta(G) \geq t$, then $G$ contains an induced star $K_{1,t}$. By Lemma~\ref{lemma:induced-star} this proves that the cop can save three vertices.
\end{proof}

The following \cref{prop:upper-bound} ensures the upper bound in Theorem~\ref{t:main}. For the proof we require a helpful lemma, by Carlson, Halloran, and Reinhart~\cite{carlson2022multi}.
\begin{lemma}\cite[Proposition~2.7]{carlson2022multi}
\label{lem:helper-lemma}
    If $G$ is any graph containing a vertex with three neighbors, then one cop can protect at least three vertices against two robbers.
\end{lemma}

\begin{proposition}
 \label{prop:upper-bound}
For all $s \geq 2$, we have $\Delta_s \leq 2\binom{s}{2} + 1$.
\end{proposition}
\begin{proof}
Let $s \geq 2$. We have to prove that for all graphs $G$ with $\Delta(G) \geq 2 \binom{s}{2}+1$, the cop can save three vertices. 
For this purpose, assume $G$ contains a vertex $v$ of degree at least $2\binom{s}{2} +1$. 
The general approach of the cop is to apply the strategy $\guard(v)$ until only two robbers remain. 

First of all, if the robbers never enter $N(v)$ all the vertices in the neighborhood will remain intact and we are done.

Otherwise, assuming $s>2$, let $i$ be the first round such that the following holds: the cop moved from vertex $v$ to catch a robber in round $i-1$ and returned to vertex $v$ in round $i$, at that point there are at most two remaining robbers, and it is their turn to move in round $i$. If $s=2$, then let $i=1$.

We claim that in this exact situation, in round $i$ where the cop has moved (and finds himself on $v$) and the robbers have not yet moved, there are at least three vertices in $N(v)$ which are not yet damaged and which are not currently occupied by a robber.

 Indeed, if $s>2$, observe that previous to that point, whenever the cop catches a robber and goes back to $v$, the $k$ robbers that are still not captured can enter at most $2k$ vertices from $N(v)$ in these two rounds. Furthermore, $v$ itself is never damaged. 
 It then follows that the number of vertices in $N(v)$ which were entered by a robber is at most $2(s-1) + 2(s-2) + \dots + 2\cdot2 = 2\binom{s}{2} - 2$. 
 Therefore at that point there are at least three vertices in $N(v)$ which are not damaged and not occupied by a robber. 
 If $s=2$, the same condition for $N(v)$ is satisfied at the start of the game.
The proof now simply follows by \cref{lem:helper-lemma}.
\end{proof}

\subsection{Lower bounds}

The goal of this subsection is to give examples of graphs with high degree, where the cop can not save three vertices. In particular, to show lower bounds for $\Delta'_s$ and $\Delta_s$, we introduce the graphs $G'$ and $G$ which are depicted in \cref{fig:lower-bound-graphs}.
\begin{figure}[htpb]
\centering
\hfill\includegraphics[scale=1,page=1]{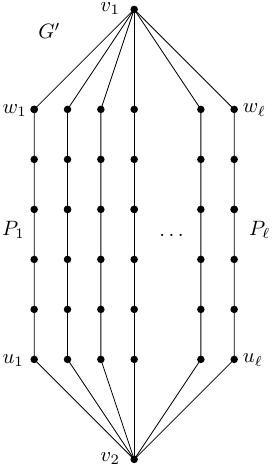}
\hfill
\includegraphics[scale=1,page=2]{lower-bound-graphs}
\hfill \phantom{a}
\caption{Graphs $G'$ and $G$ used for proving the lower bounds on $\Delta'_s$ and $\Delta_s$.}
\label{fig:lower-bound-graphs}
\end{figure}
The graph $G'$ consists of vertices $v_1$ and $v_2$ and a total of $\ell$ internally vertex-disjoint paths $P_1,\dots, P_{\ell}$ of length 7 from $v_1$ to $v_2$. If $N(v_1) = \set{w_1,\dots,w_{\ell}}$ and $N(v_2) = \set{u_1,\dots,u_{\ell}}$, then $G$ is created from $G'$, with an additional assumption that $\ell$ is even, by adding the edges $\set{w_1,w_2}, \set{w_3,w_4}, \dots, \set{w_{\ell-1}, w_{\ell}}$ and  $\set{u_1,u_2}, \set{u_3,u_4}, \dots, \set{u_{\ell-1}, u_{\ell}}$.

We now proceed to prove some helpful auxiliary statements which establish that the robbers can achieve certain goals when playing on the graphs $G$ and $G'$. A \emph{great cycle} in $G$ is a cycle of length 14 which contains both $v_1$ and $v_2$ and has no chord. In other words, a great cycle in $G'$ is created by any two paths $P_i, P_j$ with $i \neq j$, and a great cycle in $G$ is created by any two paths $P_i, P_j$ with $\lceil i/2 \rceil \not= \lceil j/2 \rceil$.

We will say that a robber is \emph{cautious} if his first priority is not to get caught by the cop. When we say that a cautious robber is playing according to a certain strategy, we mean that he plays according to the strategy unless that takes him within the cop's neighborhood, in which case he suspends the strategy and avoids getting caught. 

\begin{lemma} \label{l:stay-on-cycle}
If a cautious robber finds himself on a great cycle $\tilde C$ of $G$ (respectively $G'$), at distance at least two from the cop's position, he can stay on $\tilde C$ indefinitely.
\end{lemma}
\begin{proof}
No matter on which vertex $v\in V(\tilde C)$ the robber is, when the cop gets onto some vertex $w \in N(v)$ there is always a vertex on $\tilde C$ that the robber can move to and stay out of the cop's reach (even if $w \not \in V(\tilde C)$).
\end{proof}

\begin{lemma} \label{l:2-on-cycle}
Assume that there are at least three robbers positioned on vertices of $G$ (respectively $G'$), such that they are all at distance at least two from the cop's position. If $\tilde C$ is a great cycle, then the robbers have a strategy such that no robber gets caught and at least two robbers can enter $\tilde C$ and stay on it indefinitely.
\end{lemma}  
\begin{proof}
First we show that two robbers, $R_1$ and $R_2$, can make sure that one of them gets onto $\tilde C$. Suppose w.l.o.g.~that $R_1$ is not further away from $v_1$ than $R_2$, then $R_1$ will cautiously move towards $v_1$, and $R_2$ will cautiously move towards $v_2$. Since the robbers start at distance two away from the cop, the cop cannot catch them in the first round. Since they move into opposite directions, the cop cannot stop both of them at the same time from moving closer to their target. Therefore at least one of the robbers will manage to reach his target, thus getting on $\tilde C$.

Now we apply this reasoning twice. First we pick arbitrary two robbers and one of them moves onto $\tilde C$. Note that he can stay on $\tilde C$ indefinitely without getting caught, by~Lemma~\ref{l:stay-on-cycle}. If no other robber is on $\tilde C$ at that point, then there exist two robbers that are not on $\tilde C$. Then one of them moves onto $\tilde C$ using the same strategy.
\end{proof}

Given a great cycle $\tilde C$ of either $G$ or $G'$ and three robbers $R_1, R_2, R_3$, we are going define a strategy $\textsc{cycle-attack}(\tilde C, R_1, R_2, R_3)$. The idea behind this strategy is that the robbers, without getting caught by the cop, either manage to damage all vertices of $\tilde C$, or do a considerable damage to the rest of the graph.

Formally, the strategy $\textsc{cycle-attack}(\tilde C, R_1, R_2, R_3)$ is defined as follows. 
All three robbers remain cautious for the whole duration of the strategy. As soon as all vertices of $\tilde C$ are damaged, the $\textsc{cycle-attack}$ is concluded.

In Stage~1, the robbers $R_1$ and $R_2$ enter the cycle $\tilde C$, on which they are going to stay for the remainder of the $\textsc{cycle-attack}$. In Stage~2, the robber $R_1$ attempts to move clockwise (CW) around $\tilde C$ while the robber $R_2$ attempts to move counter-clockwise (CCW) around $\tilde C$.

If during Stage~2 the robbers $R_1$ and $R_2$ are both halted while failing to damage all the vertices of $\tilde C$, Stage~3 starts. In Stage~3 the robbers $R_1$ and $R_2$ now swap their directions, $R_1$ going CCW on $\tilde C$ and $R_1$ going CW on $\tilde C$ for the remainder of the $\textsc{cycle-attack}$.
If after that they again both stop their rotational movement before all the vertices of $\tilde C$ are damaged, Stage~4 starts and the robber $R_3$ comes into play. Suppose that w.l.o.g.~at that point the cop is closer to $v_2$ than to $v_1$. The robber $R_3$ goes to $v_1$, and then for every $i \in \set{1,\dots,\ell}$ he attempts to move down the path $P_i$ from $v_1$ towards $v_2$ for as long as possible while staying cautious, returning to $v_1$ between each two paths. Meanwhile, the robbers $R_1$ and $R_2$ remain on the same task as in Stage~3. Once the robber $R_3$ has executed this strategy for all the paths $P_i$ for $i \in \set{1,\dots,\ell}$, Stage~4 and the whole $\textsc{cycle-attack}$ is concluded.

\begin{lemma} \label{l:attack}
Let $\tilde C$ be a great cycle of $G$ (respectively $G'$). If there are at least three robbers positioned at distance at least two from the cop's position, then they can execute the strategy $\textsc{cycle-attack}(\tilde C, R_1, R_2, R_3)$ in finitely many moves without being caught.
\end{lemma}

\begin{proof}
Stage~1 can be completed by Lemma~\ref{l:2-on-cycle}. If in Stage~2 the robbers $R_1$ and $R_2$ stop moving around the cycle, that means that the cop is at distance at most two from each of them. Note that at most three vertices of $\tilde C$ can be between them at that point. 

Once Stage~3 starts the robbers $R_1$ and $R_2$ will clearly be able to meet each other on $\tilde C$ after changing the orientation of their rotational movement. If after that they are both halted, at most three vertices on $\tilde C$ can remain undamaged. Note that from that point on, if the cop wishes to protect at least one of the three remaining undamaged vertices of $\tilde C$, the cop must remain tied up within distance one from at least one of those three vertices to the end of the $\textsc{cycle-attack}$.
 
During Stage~4, as long as the cop remains at distance at most one to at least one undamaged vertex of $\tilde C$ he will not be able to enter the neighborhood of $v_1$ (assuming w.l.o.g.~that $v_1$ was chosen in Stage~4). Hence, the robber $R_3$ will be able to repeatedly find an unobstructed route to $v_1$, before starting the traverse for each of the paths $P_i$.

Clearly, the duration of each of the stages is limited and thus the $\textsc{cycle-attack}$ must conclude.
\end{proof}

\begin{lemma} \label{l:neighborhood}
If $\textsc{cycle-attack}(\tilde C, R_1, R_2, R_3)$ is performed on $G$ (respectively $G'$) for every great cycle $\tilde C$, then all of its undamaged vertices are contained in the closed neighborhood of a single vertex.
\end{lemma}

\begin{proof}
Let us assume that $\textsc{cycle-attack}(\tilde C, R_1, R_2, R_3)$ was performed on $G$ or $G'$ for every great cycle $\tilde C$.

Suppose first that a vertex $x$ that is neither in $N[v_1]$ nor in $N[v_2]$ remains undamaged. Due to the symmetry of the base graph w.l.o.g.~we can assume that $x \in V(P_1)$. As the closed neighborhood of $x$ belongs to $V(P_1)$, the cop must have stayed on $P_1$ in Stage 4 of every $\textsc{cycle-attack}(\tilde C, R_1, R_2, R_3)$ where $\tilde C$ contained $P_1$. If during any of these implementations of Stage 4, the cop got to either $u_1$ or $w_1$, then all the undamaged vertices of the base graph must be, respectively, either in $N[u_1]$ or in $N[w_1]$. On the other hand, if in all  those implementations of Stage 4, the cop remained out of both $N[v_1]$ and $N[v_2]$, then all the undamaged vertices of the base graph must be consecutive vertices of $P_1$, at most three of them, which are clearly within a closed neighborhood of a vertex.

The case that remains to be considered is when all of the undamaged vertices are in $N[v_1] \cup N[v_2]$. Clearly, it cannot be that there are undamaged vertices in both $N[v_1]$ and $N[v_2]$, as the cop cannot simultaneously be at both of the neighborhoods during the $\textsc{cycle-attack}$'s. Hence, the undamaged vertices are either all within $N[v_1]$ or all within $N[v_2]$.
\end{proof}

We are now ready to prove the two main results of this section. The first one gives the lower bound in Theorem~\ref{t:main-prime}. Recall that this theorem concerns the triangle-free setting.

\begin{proposition}
For all $s \geq 2$, we have $\Delta_s' \geq \binom{s}{2} + 2$.
\end{proposition}

\begin{proof}
First, let $s = 2$. Consider the triangle-free graph $G' = C_n$ for some $n \geq 5$. Then $G'$ has maximum degree $2 = \binom{s}{2} + 1$. We need to show that for the game with $s=2$ robbers and one cop, the cop cannot save three vertices. Indeed, this is true: One robber can cautiously go clockwise, while the other cautiously goes counterclockwise. At the end, both robbers are still alive while only three vertices remain undamaged. Then one of the robbers sacrifice themselves to get caught, while the other damages one additional vertex. 

We now generalize this argument to all $s > 2$.
We will show that $s$ robbers have a strategy to damage all but at most two vertices, when playing on the triangle-free graph $G'$ from \cref{fig:lower-bound-graphs} with $\ell = \binom{s}{2}+1$ paths. 

If $s>2$, an arbitrarily chosen three robbers $R_1$, $R_2$, $R_3$ will first perform the $\textsc{cycle-attack}(\tilde C, R_1, R_2, R_3)$ for every great cycle $\tilde C$ of $G'$. By Lemma~\ref{l:neighborhood}, all of the undamaged vertices are contained in the closed neighborhood of a single vertex. Lemma~\ref{l:attack} ensures that the three involved robbers will not be caught in the process, while the remaining robbers can clearly remain cautious.

Let us first suppose that all the undamaged vertices are in the closed neighborhood of either $v_1$ or $v_2$, w.l.o.g.~assume they are all in $N[v_2]$. That means there are at most $\binom{s}{2}+2$ undamaged vertices. For the remainder of the game, the $k$ robbers that are not captured yet repeatedly perform the following strategy that we call $\textsc{all-out-attack}(k)$.

We start with the \emph{preparation phase}, whose aim is to keep the cop occupied so that all of the robbers can eventually reach $v_1$.
A great cycle $\tilde C$ containing two undamaged vertices $u_i$ and $u_j$ is chosen, and some three robbers $R_1, R_2, R_3$ perform the $\textsc{cycle-attack}(\tilde C, R_1, R_2, R_3)$ on the cycle $\tilde C$. If they manage to damage either $u_i$ and $u_j$, the preparation phase is restarted with another great cycle. Note that such restarts are possible as long as there are at least two undamaged vertices among $u_1,\dots, u_\ell$.
Otherwise, the only way that both $u_i$ and $u_j$ remain undamaged is that the cop occupies $v_2$ and is surrounded by the two robbers, say $R_1$ and $R_2$, that went around $\tilde C$.

The strategy continues with all the robbers, except $R_1$ and $R_2$ (who remain pinning the cop on $v_2$), gathering on $v_1$. Again, if at any point the cop leaves $v_2$ one of the two vertices $u_i$ and $u_j$ gets damaged by either $R_1$ or $R_2$,  the preparation phase is restarted. Otherwise, all the other robbers will manage to gather on $v_1$. Once that is done, $R_1$ and $R_2$ also move to $v_1$ (which is possible, since the cop occupies $v_2$). Note that this way $R_1$ and $R_2$ maintain a distance of at least two from the cop.

Once all the robbers are on $v_1$ the \emph{attack phase} starts. Each of the robbers picks a different path $P_i$ such that $u_i$ is undamaged, and they simultaneously move down their picked paths towards $u_i$. Even if a robber gets within distance two from the cop, the robber proceeds towards $u_i$ risking getting caught. Each robber that reaches its designated $u_i$ becomes cautious immediately, and the $\textsc{all-out-attack}(k)$ is complete.

Note that in the preparation phase there are no caught robbers. Furthermore, each time a preparation phase is restarted, at least one of the undamaged neighbors of $v_2$ is damaged. Hence, the number of restarts is limited.

Now during the attack phase with $k$ robbers at most one robber gets caught, as once the cop catches one robber performing the attack he cannot reach another one before the attack is complete and the remaining robbers become cautious.
That means that at least $k-1$ undamaged vertices in $N[v_2]$ get damaged during the attack. 

All in all, the $\textsc{all-out-attack}(k)$ will get completed for $k=s, s-1, \dots, 2$, damaging at least $(s-1)+(s-2)+\dots+1=\binom{s}{2}$ undamaged vertices. That means all but $\binom{s}{2}+2 - \binom{s}{2} = 2$ vertices are damaged, and we are done.

In the case when all the undamaged vertices are in the closed neighborhood of a vertex $x$ that is neither $v_1$ nor $v_2$, the situation is much simpler. It is enough to perform an attack analogous to the all out attack above, but with just two robbers on two internally disjoint paths towards $x$, and the number of undamaged vertices will get reduced to 2.
\end{proof}

In the following, we prove the lower bound in Theorem~\ref{t:main}. Recall that this theorem is concerned with the case of general, not necessarily triangle-free graphs.

\begin{proposition}
For all $s \geq 2$, we have $\Delta_s \geq 2\binom{s}{2} - 3$.
\end{proposition}

\begin{proof}
The statement trivially holds for $s=2$ and $s=3$, so from now on we can assume that $s \geq 4$. We will show that $s$ robbers have a strategy to damage all but at most two vertices, when playing on $G$ with $\ell = 2\binom{s}{2}-4$, for which we have $\Delta(G)=2\binom{s}{2} - 4$. 

An arbitrarily chosen three robbers $R_1$, $R_2$, $R_3$ will first perform the $\textsc{cycle-attack}(\tilde C, R_1, R_2, R_3)$ for every great cycle $\tilde C$ of $G$. By Lemma~\ref{l:neighborhood},  all of the undamaged vertices are contained in the closed neighborhood of a single vertex. Lemma~\ref{l:attack} ensures that the three involved robbers will not be caught in the process, while the remaining robbers can clearly remain cautious.

Let us first suppose that all the undamaged vertices are in the closed neighborhood of either $v_1$ or $v_2$, w.l.o.g.~assume they are all in $N[v_2]$. That means there are at most $2\binom{s}{2}-3$ undamaged vertices. From that point on, the $k$ robbers that are not captured yet repeatedly perform the following strategy that we call $\textsc{all-out-attack-2}(k)$, until there are only three robbers left.

In the \emph{preparation phase} a great cycle $\tilde C$ containing two undamaged vertices $u_i$ and $u_j$ is chosen, and some three robbers $R_1, R_2, R_3$ perform the $\textsc{cycle-attack}(\tilde C, R_1, R_2, R_3)$ on the cycle $\tilde C$. If they manage to damage either $u_i$ and $u_j$, the preparation phase is restarted with another great cycle. Note that such restarts are possible as long as there are at least three undamaged vertices among $u_1,\dots, u_\ell$. 
Otherwise, the only way that both $u_i$ and $u_j$ remain undamaged is that the cop is pinned on $v_2$ by the two robbers, say $R_1$ and $R_2$, that went around $\tilde C$.

The strategy continues with all the robbers, except $R_1$ and $R_2$ (who remain pinning the cop on $v_2$), gathering on $v_1$. Again, if at any point the cop leaves $v_2$ one of the two vertices $u_i$ and $u_j$ gets damaged by either $R_1$ or $R_2$, and the preparation phase is restarted. Otherwise, all the other robbers manage to gather on $v_1$. Once that is done, $R_1$ and $R_2$ also move straight towards $v_1$, thus maintaining a distance of at least two from the cop.

Once all the robbers are on $v_1$ the \emph{attack phase} starts. Each of the robbers picks a different pair of paths $P_{2i-1}$, $P_{2i}$ such that $u_{2i-1}$ and $u_{2i}$ are not both damaged (one of them might have been damaged during any of the previous preparation phases), and then they simultaneously move down $P_{2i-1}$ towards $u_{2i-1}$. Even if a robber gets within distance two from the cop, the robber proceeds towards $u_{2i-1}$ risking getting caught. Crucially, when a robber gets to distance one from the cop, he will not move onto the cop's vertex but rather wait, maintaining a distance one from the cop. If the cop does not get to the robber's vertex but decides to move towards $v_2$, the robber will follow tightly at distance one.
Every robber that gets to $u_{2i-1}$, then visits $u_{2i}$ and immediately after that gets cautious. Once all the robbers that are not caught in the process complete this task, the $\textsc{all-out-attack-2}(k)$ is finished.

Note that in the preparation phase there are no caught robbers. Furthermore, each time a preparation phase is restarted, at least one of the undamaged neighbors of $v_2$ is damaged. Hence, the number of restarts is limited.

Now during the attack phase with $k$ robbers at most one robber gets caught, as once the cop catches one robber performing the attack he cannot reach another one before the attack is complete and all the remaining robbers became cautious.
That means that at least $k-1$ pairs of vertices in $N[v_2]$ get damaged during the attack. (To be a bit more precise here, it is possible to associate to the attack phase $k-1$ pairs of vertices, such that different attack phases are associated with disjoint set of pairs, and after each attack we can be sure that all vertices from the associated pairs are in a damaged state -- either damaged before or during that attack phase.) 

All in all, the $\textsc{all-out-attack-2}(k)$ will get completed for $k=s, s-1, \dots, 4$, damaging at least $2(s-1)+2(s-2)+\dots+2\cdot (4-1)=2\binom{s}{2}-6$ vertices. That means all but $2\binom{s}{2}-3 - 2\binom{s}{2} - 6 = 3$ vertices are damaged.

The final three undamaged vertices are forming a triangle. The remaining three robbers can target one undamaged vertex each, approaching them on three disjoint paths. If they remain cautions on the approach, one more vertex will clearly get damaged, completing the proof in this case.

In the case when all the undamaged vertices are in the closed neighborhood of a vertex $x$ that is neither $v_1$ nor $v_2$, we again have a considerably simpler situation to analyze. No matter what vertex $x$ is, at most two coordinated attacks by three robbers towards the neighbors of $x$ are needed to reduce the total number of undamaged vertices to at most two.
\end{proof}

\section{Two cops and two robbers on paths and cycles}
\label{sec:two-cops-two-robbers}
Here we study the game played by two cops against two robbers, on two particular families of graphs -- cycles and paths. Eventually, we provide the proofs of Theorem~\ref{thm:damage-cycle} and Theorem~\ref{thm:damage-path}.

\subsection{Tools}

We start by proving several general results which will be utilized to establish the exact damage number on paths and  cycles. In all of them it is assumed that each of the cops and robbers has a certain starting position, and either the cops or the robbers follow a certain behavior. Note that, as usual, the cops have the first move. 
\begin{figure}[thpb]
    \centering
    \includegraphics[scale=1.1, page=2]{path-lemma}
    
    \vspace*{0.3cm}
    
    \includegraphics[scale=1.1, page=1]{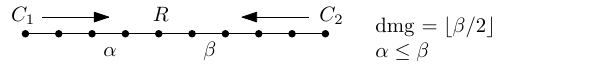}
    
    \vspace*{0.3cm}
    
    \includegraphics[scale=1.1, page=3]{path-lemma}
    
    \caption{Positions of cops and robbers in \cref{obs:path,lemma:path-lemma,lemma:path-lemma-2} with $\alpha = 4$ and $\beta=5$ and $\gamma = 9$. The number of damaged vertices  equals $\lfloor \gamma/2 \rfloor = 4$ in the first case, $\lfloor \beta/2 \rfloor = 2$ in the second case, and $\lfloor (2\alpha + \beta)/2 \rfloor = 6$ in the last case.}
    \label{fig:path-lemma}
\end{figure}

\begin{observation}
\label{obs:path}
If a cop and a robber are placed at the ends of a path of length $\gamma$, see~\cref{fig:path-lemma} (top), and move towards each other each time step, then exactly $\lfloor \gamma/2 \rfloor$ vertices are damaged.
\end{observation}
\begin{proof}
By case distinction whether $\gamma$ is even or odd. Observe that if the robber enters a vertex and gets caught immediately after, the vertex stays undamaged.
\end{proof}

\begin{lemma}
\label{lemma:path-lemma}
We assume two cops and one robber play on a path, and their starting positions are such that the robber starts between the two cops and has distance $\alpha$ and $\beta$ to the two cops, respectively, see~\cref{fig:path-lemma} (middle). 

If the cops move towards the robber each time step, then at most $\max\set{\lfloor \alpha/2 \rfloor, \lfloor \beta/2 \rfloor}$ vertices are damaged. Furthermore, if $\alpha,\beta \neq 1$, then the robber can also achieve this bound.
\end{lemma}
\begin{proof}
Assume that from left to right we have the cop $C_1$, the robber $R$ and the cop $C_2$.
In every round, the robber can make a left step, a right step, or pass his move and thereby remain on his vertex. Clearly, passing offers no advantage since it results in a total damage that is lower. 
(Observe that even if the case appears that the robber is surrounded by two cops and both a left step and a right step would enter a vertex with a cop, passing the move still does not damage more vertices than moving.) 
Therefore we can assume that every move of the robber is a left or a right step. We say that the robber turns around if he switches direction from left to right, or from right to left in two successive rounds. 
We claim that in an optimal strategy, the robber turns around either one or zero times. Indeed, it is easy to check that if the robber would turn around two or more times, there is a strategy damaging the same number of vertices with less turns.

Now assume the robber makes zero turns and goes left in every round. By \cref{obs:path} exactly $\lfloor \alpha/2 \rfloor$ vertices get damaged. Likewise, if the robber goes right, he will damage $\lfloor \beta/2 \rfloor$ vertices.

Finally, assume the robber makes $c_1$ steps to the left, turns around, and makes $c_2$ steps to the right before getting caught. First, if $c_1 \geq c_2$, then this strategy is dominated by the strategy to always move left. 
Second, observe that if $c_1 < c_2$, then the robber damages exactly $c_2$ vertices. 
However, while $R$ moved $c_1$ vertices to the left, the cop $C_2$ also moved $c_1$ vertices to the left. Therefore, this strategy is analogous to the strategy with zero turns. Hence exactly  $\lfloor \beta/2 \rfloor$ vertices are damaged. Finally, if $R$ first goes to the right and then turns around to the left, by the same argument at most $\lfloor \alpha/2 \rfloor$ vertices get damaged.

To show that this bound can always be attained when $\alpha,\beta \neq 1$, observe that the robber can simply check which of the $\alpha$ and $\beta$ is larger, and move in the corresponding direction. However, if $\alpha=1$ or $\beta=1$, the closer cop immediately catches him (as the game always starts with a cop move). 
\end{proof}

\begin{lemma}
\label{lemma:path-lemma-2}
Assume one cop and two robbers play on a path of length $\alpha+\beta$, and their starting positions are such that the cop starts between the two robbers and has distance $\alpha$ and $\beta$ to the two robbers, respectively, with $\alpha \leq \beta$, see~\cref{fig:path-lemma} (bottom). Furthermore assume that the robbers move towards the cop without ever stopping. 
If the cop plays optimally, then exactly $\lfloor (2\alpha + \beta)/ 2 \rfloor$ vertices are damaged. 
\end{lemma}
\begin{proof}
We assume that from left to right we have the robber $R_1$, the cop $C$ and the robber $R_2$ in this order, and we assume that $C$ has distance $\alpha$ to $R_1$ and distance $\beta$ to $R_2$. As we assumed that the robbers move towards the cop each time step, it is clear that the cop should either move first towards $R_1$ to catch it and then catch $R_2$ afterwards, or the other way around. Let us first assume that the cop goes to catch $R_1$ before $R_2$. We distinguish the case whether $\alpha$ is even or odd:

\textbf{$\alpha$ is even:}
In this case, after $\alpha/2$ cop moves, when the robber does his $(\alpha/2)$-st move, the robber $R_1$ steps onto the cop and is caught. This results in $R_1$ damaging $\alpha/2$ vertices in total. We can now ask how many vertices did $R_2$ damage? Note that $C$ is $\alpha/2$ left steps away from its starting position, so after another $\alpha/2$ right steps, $C$ will be back at the starting position. We can apply \cref{obs:path} to a path of length $2\cdot \alpha/2 + \beta$, so we see that $R_2$ damaged $\lfloor (\alpha + \beta) / 2 \rfloor$ vertices. As $\alpha/2$ is an integer, together the robbers damaged exactly $\lfloor (2\alpha + \beta) / 2 \rfloor$ vertices.

\textbf{$\alpha$ is odd:} Consider the situation after $(\alpha-1)/2$ steps: In this situation, $R_1$ and $C$ are on adjacent vertices. Now there seem to be two reasonable strategies for the cop: Either the cop could immediately catch $R_1$ and prevent $R_1$ from damaging the vertex $R_1$ is currently on. This way the cop would prevent the robber from damaging a vertex, but would also move one unit away from the starting position. The second option would be to pass for a single round, and let the robber $R_1$ walk onto the cop. In this strategy, robber $R_1$ damages one more vertex, but $C$ is one unit closer towards $R_2$.

First we examine the strategy where $R_1$ is caught immediately: There $R_1$ damages $(\alpha-1)/2$ vertices. Now $C$ is $(\alpha + 1)/2$ left steps away from the starting position. By an analogous argument as in the case where $\alpha$ is even, the robber $R_2$ damages $\lfloor (\alpha + 1 + \beta)/2 \rfloor$ vertices before getting caught. Together, exactly $\lfloor (2\alpha + \beta) / 2 \rfloor$ vertices are damaged. 

Secondly, we examine the strategy where $C$ passes for a single round. There $R_1$ damages $(\alpha+1)/2$ vertices and $R_2$ damages $\lfloor (\alpha + \beta)/2 \rfloor$ vetrices (because \cref{obs:path} is applied to a path of length $2\cdot(\alpha-1)/2 + 1 + \beta$). In total, $\lfloor (2\alpha + \beta + 1) / 2 \rfloor$ vertices are damaged. Hence we see that the strategy of passing is inferior to the other strategy.

This concludes the case distinction. Finally, if $C$ goes to catch $R_2$ first, by symmetry we see that $\lfloor (2\beta + \alpha) / 2 \rfloor$ vertices are damaged. This is inferior to the strategy of catching $R_1$ first. Therefore we conclude that under optimal play by the cop, exactly $\lfloor (2\alpha + \beta) / 2 \rfloor$ vertices are damaged.
\end{proof}

\subsection{Two cops and two robbers on a cycle}
\label{subsection:cycle}

Here we investigate the damage number $\dmg(C_n;2,2)$ of the cycle $C_n$ with two cops and two robbers. Our goal is to eventually, in several steps, prove Theorem~\ref{thm:damage-cycle}.

At first glance, one might think that the best strategy for $t=2$ cops should be to initially position themselves in such a way that the cycle is divided into two pieces of roughly equal length, but interestingly that is not the case. Our analysis shows that the cops' initial positions should divide the cycle into two parts of length roughly $n/3$ and $2n/3$ each. The rest of the section is devoted to establishing the exact number $\dmg(C_n;2,2)$ for all $n \geq 4$.

\paragraph{Notation.} We will encounter the same setting multiple times, so we establish a standard notation. In this setting we have the cycle on $n$ vertices, and the cops $C_1$ and $C_2$ have already chosen their starting positions such that the cycle is divided into two paths $P_A$ and $P_B$ of length $a$ and $b$, such that $a + b = n$. We always assume $a \leq b$. W.l.o.g., the cycle is arranged so that we have in clockwise order the cop $C_1$, the path $P_A$, the cop $C_2$ and the path $P_B$, see \cref{fig:circle-notation}. Let $x_0,\dots,x_a$ be the $a+1$ vertices of $P_A$ and $y_0,\dots,y_b$ be the $b+1$ vertices of $P_B$ in clockwise order, i.e., cop $C_1$ starts at vertex $x_0=y_b$ and cop $C_2$ starts at vertex $y_0=x_a$. First the cops choose the starting position and hence the values of $a$ and $b$. After that, the robbers choose their starting position, followed by the cop's first move.
\begin{figure}[thpb]
    \centering
    \includegraphics[scale=1.1]{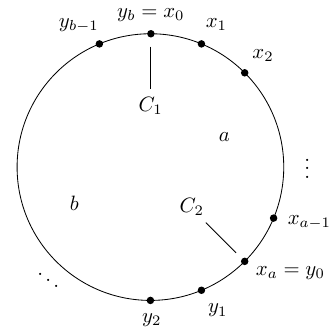}
    \caption{Notation used in \cref{subsection:cycle}}
    \label{fig:circle-notation}
\end{figure}

\begin{lemma}
\label{lemma:cops-protect-cycle}
 Consider the cycle $C_n$ for $n \geq 4$. If the cops choose initial positions such that $b \geq 6$, where $a,b$ are defined as above, see Figure~\ref{fig:circle-notation}, then the cops have a strategy to ensure that at most $\max\set{b-3,\lfloor (n + a)/2 \rfloor - 3}$ vertices get damaged.
\end{lemma}
\begin{proof}
The two cops subdivide the cycle of length $n = a+b$
 into two paths $P_A, P_B$ of length $a$ and $b$ each. Note that the direction of clockwise and counterclockwise is defined on the cycle, and by extension it is also defined on $P_A,P_B$.
 If both robbers choose to begin on the same path, then at most $b-3$ vertices are damaged (because the path $P_B$ has $b+1$ vertices, and the four distinct vertices $y_0,y_1,y_{b-1},y_b$ can never get damaged if the cops immediately move towards the robbers). Therefore we assume for the rest of the proof that one robber $R_A$ starts on the path $P_A$ and one robber $R_B$ starts on the path $P_B$, and we prove that in this case at most $\lfloor (n + a)/2 \rfloor - 3$ vertices get damaged.

 As soon as the game starts, the cops check whether a robber chose $x_1, x_{a-1}, y_1$, or $y_{b-1}$ as starting position. In this case they immediately catch this robber and then both cops approach and catch the remaining robber. It is easy to see that this way at most $\lfloor b/2 \rfloor$ vertices get damaged (one can use \cref{lemma:path-lemma} with $\alpha = 2, \beta = b$). Using the assumption $b \geq 6$ we have $\lfloor b/2 \rfloor \leq b - 3$. Therefore we can assume for the remaining proof that no robber starts at one of those four vertices. 
 
We distinguish two cases: when $a$ is odd and when $a$ is even.

\textbf{$a$ is odd:} The cops employ the following strategy. In Phase 1, cop $C_1$ goes $(a-1)/2$ steps in clockwise direction, while $C_2$ goes $(a-1)/2$ steps in counterclockwise direction. Observe that after this  maneuver, the cops are on adjacent vertices and so robber $R_A$ is caught. After this phase 2 starts, both cops turn around and cop $C_1$ goes in counterclockwise direction while cop $C_2$ goes in clockwise direction until the two cops meet again. We claim that the robbers damage at most $\lfloor (n + a)/2 \rfloor -3$ vertices.

Robber $R_A$ gets caught after $(a-1)/2$ cop moves, so he clearly damaged at most $(a-1)/2 - 1$ vertices. How many vertices can $R_B$ damage before being caught? Consider the following change of perspective: Instead of cops $C_1$ and $C_2$ first moving away from $R_B$ for $(a-1)/2$ steps each in Phase 1 and then moving towards $R_B$ in Phase 2, we imagine they are located on the endpoints of a path of length $b + 4(a-1)/2 = b +2a - 2$ and right away move towards $R_B$. This way, after the initial $(a-1)/2$ steps they will be at the same relative positions as they were right after Phase 1 (at distance $n-1$, with a robber between them), but during those initial $(a-1)/2$ steps the robber has more freedom than in the actual scenario. Hence, an upper bound for damage in the imagined setting is an upper bound for the actual setting as well.

Now we can apply \cref{lemma:path-lemma} in the imagined setting. Note that $R_B$ does not start on $y_1,y_{b-1}$ and so the maximal distance $R_B$ can have from the endpoint of the path is $(b - 2) + (a - 1) = n-3$. By the lemma, $R_B$ damages at most $\lfloor (n-3)/2 \rfloor$ vertices. Using the fact that $(a-1)/2$ is an integer, we have that the robbers combined cause at most $(a-1)/2 - 1 + \lfloor (n-3)/2 \rfloor = \lfloor (n + a)/2 \rfloor - 3$ damage.

\textbf{$a$ is even:} The cops employ the following strategy. First, $C_1$ moves $a/2 - 1$ steps clockwise and $C_2$ moves $a/2 - 1$ steps anticlockwise. After this, cop $C_1$ occupies vertex $x_{a/2-1}$, cop $C_2$ occupies vertex $x_{a/2+1}$ and it is possible that the robber $R_A$ could be waiting on $x_{a/2}$, not moving. Which of the cops should capture the robber? This is an important decision because the cop that does not capture the robber can immediately move towards the other robber $R_B$ and capture $R_B$ one round earlier, preventing one vertex from being damaged. 

To resolve this issue, the cops consider the vertex $v_B$ which the robber $R_B$ chose as its initial starting position. If $v_B = y_2$ or $v_B = y_3$, then $C_2$ goes one step counter-clockwise and captures $R_A$, and $C_1$ also goes one step counter-clockwise already heading towards the other robber. In all other cases $C_1$ goes one step clockwise and captures $R_A$, and $C_2$ also goes one step clockwise. After that, the two cops proceed like in Case 1, i.e., cop $C_1$ goes counter-clockwise and $C_2$ goes clockwise until the remaining robber is caught.

Because robber $R_A$ is caught after $a/2$ cop moves, it causes at most $a/2 - 1$ damage. How much damage can $R_B$ cause? We again change our perspective and imagine a path consisting of the concatenation of the three paths $Q_1$ and $P_B$ and $Q_2$, where in the case $v_B \in \set{y_2, y_3}$ the path $Q_1$ has length $2a-2$ and path $Q_2$ has length $2a$, and in all other cases path $Q_1$ has length $2a$ and path $Q_2$ has length $2a-2$. The path $P_B$ is as previously defined.

We can imagine that cop $C_1$ is at the far end of $Q_1$ and the cop $C_2$ is at the far end of $Q_2$ and they are both moving towards the robber $R_B$. As in the previous case, it is straightforward to conclude that an upper bound obtained in this imagined situation will be an upper bound for the actual situation. We now apply \cref{lemma:path-lemma} to give a bound on the number of vertices damaged by $R_B$.
For every $i \in \set{2,\dots,b-2}$, we consider the case where $v_B = y_i$ and we let $d_1(i)$ in this case denote the distance of the robber's starting position $v_B$ to the cop $C_1$ on the new path $(Q_1,P_B,Q_2)$. Let $d_2(i)$ be defined analogously for cop $C_2$. We then have
\begin{align*}
    d_1(i) &= \begin{cases}
                    b - i + a - 2, & \text{if } i \in\set{2,3}\\
                    b - i + a, & \text{otherwise,}
                \end{cases} \\
    d_2(i) &= \begin{cases}
                    i + a, & \text{if } i \in\set{2,3}\\
                    i + a - 2, & \text{otherwise.}
                \end{cases}
\end{align*}
Furthermore, by \cref{lemma:path-lemma}, the robber $R_B$ damages at most $\max\{\lfloor d_1(i)/2 \rfloor, \lfloor d_2(i)/2 \rfloor\}$ vertices. 
We claim that both $\lfloor d_1(i)/2 \rfloor, \lfloor d_2(i)/2 \rfloor \leq \lfloor n/2 \rfloor - 2$, for all values of $i$.  
Indeed, observe that in the cases $i=4,\dots,b-2$, we have
\begin{align*}
    &d_1(b-2) \leq  \dots \leq d_1(4) = b-4+a = n-4,\\[2mm]
    &d_2(4) \leq \dots \leq d_2(b-2) = a + b -4 = n - 4.
\end{align*}
Note that $d_2(b-2) = a + b - 4$, because $b-2 \geq 4$, by assumption on $b$. Furthermore in the cases $i = 2,3$, we have
\begin{align*}
    d_1(3) &\leq d_1(2) = b-2+a-2 = n-4,\\[2mm]
    \left\lfloor \frac{d_2(2)}{2} \right\rfloor &\leq  \left\lfloor \frac{d_2(3)}{2} \right\rfloor = \left\lfloor \frac{a + 3}{2} \right\rfloor \leq \left\lfloor \frac{n - 4}{2} \right\rfloor.
\end{align*}
Note that the last inequality is true, because in the case $b \geq 7$ we have $a + 3 \leq n - 4$, while in the case $b = 6$, we have $a + 3 \leq n - 3$, but also $n$ is even and so $\lfloor (n - 3)/2 \rfloor = \lfloor (n - 4)/2 \rfloor$. In total, robber $R_B$ causes at most $\lfloor(n-4)/2\rfloor$ damage, and both robbers combined cause at most $a/2 - 1 + \lfloor(n-4)/2\rfloor = \lfloor (n+ a)/2 \rfloor - 3$ damage. This concludes the proof.
\end{proof}

We note that if the assumption $b \geq 6$ is dropped from the lemma statement, then it is not true anymore. Indeed, in the case $b=5$ and $a = 4$, a clever robber $R_B$ can start on vertex $y_2$ and observe which cop $C_1$ or $C_2$ catches  $R_A$ to adjust its direction accordingly. This way, the two robbers damage $\lfloor (n + a)/2 \rfloor - 2$ vertices.

\begin{lemma}
\label{lemma:cycle-dmg-upper-bound}
 For all $n \geq 9$, the cops have a strategy ensuring that at most $\lfloor 2n/3 + 1/2 \rfloor - 3$ vertices get damaged. In other words, $\dmg(C_n;2,2) \leq \lfloor 2n/3 + 1/2 \rfloor - 3$ for $n \geq 9$.
\end{lemma}
\begin{proof}
The idea is that the cops position themselves to divide the cycle into parts of size roughly 1/3 and 2/3. The exact size depends on whether $n = 0,1,2 \text{ (mod 3)}$.

\textbf{Case 1: $n = 3k$ for some $k \in \N, k \geq 3$}. \\
We have $\lfloor 2n/3 + 1/2 \rfloor - 3 = 2k - 3$. The cops choose initial positions such that $a = k$ and $b = 2k$, then $b \geq 6$, and \cref{lemma:cops-protect-cycle} is applicable, and hence  the number of damaged vertices is at most $\max\set{b-3, \lfloor (n + a)/2 \rfloor - 3} = \max\set{2k-3,2k-3} = 2k-3$.

\textbf{Case 2: $n = 3k + 1$ for some $k \in \N, k \geq 3$}. \\
We have $\lfloor 2n/3 + 1/2 \rfloor - 3 = 2k - 2$. The cops choose initial positions such that $a = k+1$ and $b = 2k$, then $b \geq 6$, and \cref{lemma:cops-protect-cycle} is applicable, and hence the number of damaged vertices is at most $\max\set{b-3, \lfloor (n + a)/2 \rfloor - 3} = \max\set{2k-3, \lfloor (4k+2)/2 \rfloor - 3} = \max\set{2k-3,2k-2} = 2k-2$.

\textbf{Case 3: $n = 3k + 2$ for some $k \in \N, k \geq 3$}. \\
We have $\lfloor 2n/3 + 1/2 \rfloor - 3 = 2k - 2$. The cops choose initial positions such that $a = k+1$ and $b = 2k+1$, then $b \geq 6$, and \cref{lemma:cops-protect-cycle} is applicable, and hence the number of damaged vertices is at most $\max\set{b-3, \lfloor (n + a)/2 \rfloor - 3} = \max\set{2k-2, \lfloor (4k+3)/2 \rfloor - 3} = \max\set{2k-2,2k-2} = 2k-2$. \par\hfill\qedhere
\end{proof}

\begin{lemma}
\label{lemma:cycle-dmg-lower-bound}
 For all $n \geq 6$, the robbers have a strategy ensuring that at least $\lfloor 2n/3 + 1/2 \rfloor - 3$ vertices get damaged. In other words, $\dmg(C_n;2,2) \geq \lfloor 2n/3 + 1/2 \rfloor - 3$ for $n \geq 6$.
\end{lemma}
\begin{proof}
The robbers have two principal strategies. By Strategy~I, they start at vertices $y_2$ and $y_{b-2}$ and walk towards each other to damage most of the path $P_B$. Under the assumption $b \geq 3$, this way $b - 3$ vertices are damaged. The second one, denoted by Strategy~II, is for robber $R_A$ to start at $x_2$ and walk clockwise without ever stopping while robber $R_B$ starts at $y_{b-2}$ and walks counterclockwise without ever stopping. As the robbers do not ever stop and both run away from cop $C_1$, it is clear that $C_1$ can never catch any robber and cop $C_2$ has to catch both robbers (as it is safe to assume that the two cops do not cross over each other). 

Therefore, under the assumption $a,b \geq 2$, we have that \cref{lemma:path-lemma-2} is applicable with $\alpha = a-2$ and $\beta = b-2$, and hence the number of damaged vertices is $\lfloor (2a + b - 6)/2 \rfloor = \lfloor (n + a)/2 \rfloor - 3$. We now make a case distinction whether $n = 0,1,2 \text{ (mod 3)}$.

\textbf{Case 1: $n = 3k$ for some $k \in \N, k \geq 3$}. \\
We have $\lfloor 2n/3 + 1/2 \rfloor - 3 = 2k - 3$. The robbers observe the cops. If $b \geq  2k$, they use Strategy~I. This damages at least $2k - 3$ vertices. Otherwise, we have $b \leq 2k - 1, a \geq k + 1$ and the robbers use Strategy~II. This damages at least $\lfloor (n + a)/2 \rfloor - 3 = \lfloor (3k + k + 1)/2 \rfloor - 3 = 2k - 3$ vertices.

\textbf{Case 2: $n = 3k + 1$ for some $k \in \N, k \geq 3$}. \\ 
We have $\lfloor 2n/3 + 1/2 \rfloor - 3 = 2k - 2$. The robbers observe the cops. If $b \geq  2k + 1$, they use Strategy~I. This damages at least $2k - 2$ vertices. Otherwise, we have $b \leq 2k, a \geq k + 1$ and the robbers use Strategy~II. This damages at least $\lfloor (3k + 1 + k + 1)/2 \rfloor - 3 = 2k - 2$ vertices.

\textbf{Case 3: $n = 3k + 2$ for some $k \in \N, k \geq 3$}. \\
We have $\lfloor 2n/3 + 1/2 \rfloor - 3 = 2k - 2$. The robbers observe the cops. If $b \geq  2k + 1$, they use Strategy~I. This damages at least $2k - 2$ vertices. Otherwise, we have $b \leq 2k, a \geq k + 2$ and the robbers use Strategy~II. This damages at least $\lfloor (3k + 2 + k + 2)/2 \rfloor - 3 = 2k - 1$ vertices.
\end{proof} 

\begin{lemma} \label{lemma:cycle-finite-cases}
 For $n \in \set{4,5,6,7,8}$, we have $\dmg(C_7;2,2) = \lfloor 2n/3 + 1/2 \rfloor - 3$. 
\end{lemma}
\begin{proof}
Using a few case distinctions, it is easy to verify $\dmg(C_4;2,2) = 0$, $\dmg(C_5;2,2) = 0$, and  $\dmg(C_6;2,2) = 1$. Consider $n = 7$. By \cref{lemma:cycle-dmg-lower-bound}, the robbers have a strategy to damage two vertices. On the other hand, the cops have the following strategy to prevent more than two damaged vertices: They choose $a = 3, b = 4$. If both robbers start on the same path, they cause at most $b - 3 = 1$ damage. If one robber starts on path $P_A$ and the other on $P_B$, note that because $a = 3$, one of the robbers is on a vertex directly next to a cop. This cop can immediately catch the robber. The other robber causes at most $\lfloor b/2 \rfloor = 2$ damage, because of \cref{lemma:path-lemma-2} applied with $\alpha = 2, \beta = b$. This shows $\dmg(C_7;2,2) = 2$.

Finally consider the case $n = 8$. By \cref{lemma:cycle-dmg-lower-bound}, the robbers have a strategy to damage two vertices. On the other hand, the cops can choose $a = 3, b = 5$. Again we have $a = 3$, so by an analogous argument, the damage the robbers cause is either at most $b - 3 = 2$, or at most $\lfloor b/2 \rfloor = 2$. This shows $\dmg(C_8;2,2) = 2$.
\end{proof}

Theorem~\ref{thm:damage-cycle} follows directly from Lemma~\ref{lemma:cycle-dmg-upper-bound}, Lemma~\ref{lemma:cycle-dmg-lower-bound} and Lemma~\ref{lemma:cycle-finite-cases}.

\subsection{Two cops and two robbers on a path}

We now consider the damage number $\dmg(P_{n+1}; 2,2)$ of a path on $n$ edges with two cops and two robbers.
We believe that with a more rigorous analysis the exact value of this number could be determined, with meticulous calculations similar to the ones performed in the previous subsection. In order not to overload the paper with technicalities we choose to prove only bounds that are an additive constant away.

We prove \cref{thm:damage-path} in two steps, first providing a strategy of the cops to protect roughly $n/3$ vertices, 
and then providing a matching strategy of the robbers, damaging roughly $2n/3$ vertices. This suffices to show the theorem.

\begin{lemma}
    For all $n \geq 1$, we have $\dmg(P_{n+1};2,2) \leq 2n/3$.
\end{lemma}
\begin{proof}
    The cops choose their starting positions such that the path is divided into roughly three equal parts, 
    and after that they never move again. 
    Since the robbers can not pass a cop, no matter where they start, at most $2n/3$ of the vertices are accessible for them to damage.
\end{proof}

\begin{lemma}
\label{lemma:path-strategies}
    We have $\dmg(P_{n+1};2,2) \geq 2n/3 - \Theta(1)$, when $n$ tends to infinity.
\end{lemma}
\begin{proof}
    The robbers observe the initial placement $C_1, C_2$ of the cops, which divide the path into three segments, $P_A$, $P_B$, $P_C$, of length $a$, $b$, $c$, respectively, 
    such that from left to right we have $P_A, C_1, P_B, C_2, P_C$ in that order, and $a+b+c = n$.
    By symmetry, we can assume that $a \leq c$, so we will make that assumption for the rest of the proof.
    We list a set of potential strategies for the robbers, depicted in \cref{fig:path-strategies}.
\begin{figure}[thpb]
    \centering
    \includegraphics[scale=1.1, page=2]{path-strategies}
    
    \vspace*{0.1cm}
    
    \includegraphics[scale=1.1, page=1]{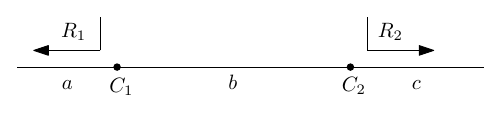}
    
    \vspace*{0.1cm}
    
    \includegraphics[scale=1.1, page=3]{path-strategies}
    
    \caption{Strategies I, II and III (from top to bottom), used by the robbers in \cref{lemma:path-strategies}.}
    \label{fig:path-strategies}
\end{figure}

    In Strategy I, robber $R_1$  initially lands directly right of $C_1$ (at a distance of 2) and moves right afterwards, 
    and robber $R_2$ positions himself directly left of $C_2$ (at a distance of 2) and moves left afterwards. 
    Strategy I ensures the robbers at least $b - \Theta(1)$ damage.
    
    In Strategy II, robber $R_1$ initially lands directly left of $C_1$ (at a distance of 2) and moves left afterwards, and robber $R_2$ positions himself directly right of $C_2$ (at a distance of 2) and moves right afterwards. 
    Strategy II ensures the robbers at least $a + c - \Theta(1)$ damage.

    In the case of either $a + c \geq 2n/3$ or $b \geq 2n/3$, the lemma is thus proven. 
    Hence we can assume for the rest of the lemma that both $a + c  < 2n/3$ and $b < 2n/3$.
    In this case, the robbers employ Strategy III, which is defined as follows.

    Robber $R_1$ positions himself directly right of $C_1$, at a distance of two, and moves right afterwards.
    However, Robber $R_2$ has a more complex task. Initially he lands directly right of $C_2$ at a distance 2, and tries to follow the movements of $C_2$, ``shadowing'' him.
    That is, if $C_2$ takes a step to the right (left, respectively), then $R_2$ also takes a step to the right (left, respectively), thus maintaining the distance 2.
    He continues this behavior until he is either forced into the right end of the path and caught by $C_2$, 
    or until $R_1$ and $C_2$ meet. 
    If $R_1$ and $C_2$ meet, then $R_2$ changes his behavior and starts moving right for as long as possible.
    
    We claim that under our assumptions Strategy III surely damages roughly $2n/3$ vertices. Indeed, consider first the case that the cop $C_2$ has as first priority to catch $R_1$. 
    Knowing that $R_2$ is shadowing him, and because $R_2$ is not caught before $R_1$, we have that the robbers damage $b + c - \Theta(1) \geq 2n/3 -\Theta(1)$ vertices. 
    Here the previous calculation follows due to $a \leq c$ and $a + c < 2n/3$, hence $a < n/3$ and $b+c > 2n/3$.
    
    In the final case we consider what happens if $C_2$ goes to catch $R_2$ first. Note that it takes $C_2$ a total number of $2c$ steps to catch $R_2$ and return to his starting position afterwards.
    Hence if $2c \geq b$, again the robbers damage $b + c - \Theta(1)$ vertices, so we are done.
    
    Lastly, if $2c < b$, then robber $R_1$ only damages $b'$ vertices for some $0 \leq b' < b$. (And robber $R_2$ still damages $c$ vertices.) 
    This $b'$ satisfies $b' = 2c + (b - b')$, since both $R_1$ and $C_2$ take the same number of steps until they meet, implying $b' = c + b/2$, so the total damage is $2c + b/2$.
    Now, note that since $a \leq c$ we have $b = n - a - c \geq n - 2c$. Hence we have
    $n - 2c \leq b < 2n/3$, which implies $c > n/6$. We come to the conclusion that up to $\Theta(1)$ vertices, the total damage is at least 
    \[2c + b/2 \geq 2c + \frac{n - 2c}{2}  = c + \frac{n}{2}> \frac{2n}{3}.\]
    
    We conclude that in all cases, one of the three described strategies ensures that the robbers can damage at least $2n/3 -\Theta(1)$ vertices.
\end{proof}

\section{Conclusion}

We studied a variant of Cops and Robbers on graphs in which the robbers damage the visited vertices. Our initial interest was with the  game with one cop against $s$ robbers, and particularly the conjecture from~\cite{carlson2022multi} stating that the cop can save three vertices from damage as soon as the maximum degree of the base graph is at least $\binom{s}{2} + 2$. We verified the conjecture and proved that it is tight with an additional assumption that the base graph is triangle free. However, without that assumption, while we were able to disprove the conjecture and provide bounds an additive constant away, the exact number remains out of our reach, leaving an enticing open problem. 

Furthermore, we initiated the study of the same game with two cops and two robbers. We determined exactly the damage number $dmg(C_n;2,2)$ when the game played on a cycle, and gave bounds for $\dmg(P_{n+1};2,2)$ an additive constant away for the game on a path. It would be interesting to determine the damage number for any other families of base graphs.

We note that, all in all, the game has been studied with one cop against a number of robbers, one robber against a number cops, and in the two-against-two setting. To our knowledge there are no results covering any other pairs of cardinalities of the cops and robbers teams.

\section*{Acknowledgments}
This work was initiated at the 19th Gremo’s Workshop on Open Problems (GWOP), which took place in Binn, Switzerland in June 2022.
We thank the organizers and all the participants of the workshop for the inspiring atmosphere, and we are particularly indebted to Nicolas Grelier and Saeed Ilchi for the fruitful discussions in the early stages of this project.

We would like to thank the anonymous referees whose comments and suggestions improved our paper.

\bibliographystyle{acm}
\bibliography{ref-damage}

\end{document}